\documentclass[12pt,a4paper]{article}
\usepackage{amssymb,amsmath,amsthm}
\usepackage{graphicx}
\usepackage{hyperref}
\usepackage{xcolor}

\newcommand\F{{\mathbf F}}

\newcommand\regregsat{\mathrm{rrsat}}
\newcommand\overregsat{\mathrm{orsat}}
\newcommand\regsat{\mathrm{rsat}}

\theoremstyle{plain}
\newtheorem{theorem}{Theorem}[section]
\newtheorem{lemma}[theorem]{Lemma}

\newtheorem{construction}[theorem]{Construction}

\newtheorem{proposition}[theorem]{Proposition}

\theoremstyle{definition}

\newtheorem{defn}[theorem]{Definition}

\newtheorem{claim}[theorem]{Claim}

\theoremstyle{plain}
\makeatletter
\newtheorem*{rep@theorem}{\rep@title}
\newcommand{\newreptheorem}[2]{%
\newenvironment{rep#1}[1]{%
 \def\rep@title{#2 \ref{##1}}%
 \begin{rep@theorem}}%
 {\end{rep@theorem}}}
\makeatother
\newreptheorem{theorem}{Theorem}
\newreptheorem{proposition}{Proposition}

\newcommand\cref[1]{Corollary~\ref{cor:#1}}

\title{Saturation problems with regularity constraints}
\author{D\'aniel Gerbner$^1$ \ \ Bal\'azs Patk\'os$^{1,2}$ \ \ Zsolt Tuza$^{1,3}$ \ \ M\'at\'e Vizer$^1$ \\ \small $^1$ Alfr\'ed R\'enyi Institute of Mathematics\\ \small $^2$ Moscow Institute of Physics and Technology\\ \small $^3$ Department of Computer Science and Systems
 Technology, University of Pannonia
}
\date{}

\begin{document}

\maketitle

\begin{abstract}
    For a graph $F$, we say that another graph $G$ is $F$-saturated, if $G$ is $F$-free and adding any edge to $G$ would create a copy of $F$. We study for a given graph $F$ and integer $n$ whether there exists a regular $n$-vertex $F$-saturated graph, and if it does, what is the smallest number of edges of such a graph. We mainly focus on the case when $F$ is a complete graph and prove for example that there exists a $K_3$-saturated regular graph on $n$ vertices for every large enough $n$. 
    
    We also study two relaxed versions of the problem: when we only require that no regular $F$-free supergraph of $G$ should exist or when we drop the $F$-free condition and only require that any newly added edge should create a new copy of $F$.
\end{abstract}

\section{Introduction}

Extremal graph theory often deals with finding the largest or smallest number of edges in $n$-vertex graphs satisfying some specified properties. The main example is
Tur\'an theory, where we look for the largest number of edges in $n$-vertex $F$-free graphs for some fixed graph~$F$.

A natural counterpart is saturation, where we look for the smallest number of edges in $n$-vertex $F$-saturated graphs. A graph $G$ is called \textit{$F$-saturated} if it is $F$-free, but adding any edge to $G$ creates a copy of $F$. For a survey on graph saturation problems, see \cite{ffs}. We mention only one result: K\'aszonyi and Tuza \cite{katu} showed that for any $F$ (in fact, for any family of graphs) the smallest number of edges in an $F$-saturated graph is at most linear in the number of vertices.

Recently, a variant of Tur\'an problems have attracted attention \cite{cvk,ct,gptv1,gptv2,tt}, where one looks for the largest number of edges in $F$-free regular graphs. Here we initiate the study of regular saturation problems. Observe that there are several different quantities one might study: the smallest number of edges in a regular $F$-saturated graph on $n$ vertices, the smallest number of edges in a regular $F$-free graph $G$ such that any regular supergraph of $G$ contains $F$ or the smallest number of edges in a regular not necessarily $F$-free graph on $n$ vertices such that any newly added edge creates a new copy of $F$. Here we study all these variants.

\begin{defn}

Let $rsat(n,F)$ denote the smallest number of edges in a regular $n$-vertex $F$-saturated  graph, if such a graph exists.

Let $rrsat(n,F)$ denote the smallest number of edges in a regular $n$-vertex $F$-free graph $G$ such that any regular $n$-vertex graph containing $G$ contains a copy of $F$. 

Let $orsat(n,F)$ denote the smallest number of edges in a regular $n$-vertex  graph $G$ such that for any non-edge $e\notin E(G)$, the graph $G'$ with $V(G')=V(G)$ and $E(G')=E(G)\cup \{e\}$ contains a copy of $F$ with $e\in E(F)$. Such graphs $G$ will be called $F$-\textit{oversaturated}.
\end{defn}

Observe that if $S_r$ denotes the star with $r$ leaves, then clearly $\regsat(2k+1,S_{2r})$ does not exist. Hence in case of $\regsat(n,F)$, the primary question whether it exists or not.

\begin{theorem}\label{K3odd}
There exists an $n_0$ such that $\regsat(n,K_3)$ exists for every $n\ge n_0$. 
\end{theorem}

Note that for small $n$ the theorem above may not hold. For example, if $n=7$ then $\regsat(n,K_3)$ does not exist, as there is no $d$-regular graph on 7 vertices if $d$ is odd, there is no $d$-regular triangle-free graph on 7 vertices if $d>3$, and it is easy to see that there is no 2-regular triangle-saturated graph on 7 vertices.

We remark that the example graph $G$ in the proof of Theorem \ref{K3odd} has quadratic many edges. This easily implies that after finding the triangle, we can find in $G$ any graph that consists of any number of bipartite components and one component which contains only one cycle, namely a triangle. Therefore, $\regsat(n,F)$ exists for those graphs $F$ also for large enough~$n$.

Let us consider now the value of $\regsat(n,F)$.
In case $F$ has a cut edge and $n$ is divisible by $|V(F)|-1$, then $n/(|V(F)|-1)$ copies of $K_{|V(F)|-1}$ form an $F$-saturated regular graph, which shows ${\displaystyle\liminf_{n \rightarrow \infty}}$ $\regsat(n,F)=O(n)$. In case $F$ does not have a cut edge, we can give a simple superlinear lower bound on
$\regsat(n,F)$, if it exists. This is in sharp contrast to the ordinary notion of saturation, where there is a linear upper bound, as we have mentioned. Before stating our theorem we introduce a notion: for any graph $F$ and an edge $e$ of $F$ we denote by $F \setminus e$ the graph on the same vertex set as $F$ and deleting the edge $e$ from its edge set.

\begin{proposition}\label{lower}~\\
\indent 
(i) Assume that every edge of $F$ is in a cycle of length at most $m+1$. Then $$\regsat(n,F)= \Omega(n^{1+1/m}).$$

(ii) If for any edge $e$ of $F$, the graph $F\setminus e$  has diameter at most $r$, then $$\regsat(n,F)= \Omega(n^{1+1/r}).$$

(iii) Given a connected graph $F$, there exists a constant $c=c_F$ such that $$\liminf_{n \rightarrow \infty} \frac{\regsat(n,F)}{n}\le c$$ 

if and only if $F$ contains a cut edge.
\end{proposition}

We can prove a subquadratic upper bound for some graphs, including cliques. Note that we do not know whether $\regsat(n,K_{s+2})$ exists for every large enough $n$, we only show a sequence of integers $n_i$ and regular $K_{s+2}$-saturated graphs on $n_i$ vertices with $o(n_i^2)$ edges.

\begin{theorem}\label{szubkvad}
For any $\varepsilon > 0$ and integer $s \ge 1$ there exists a $d$-regular graph $F$ on $n$ vertices that is $K_{s+2}$-saturated and $\frac{d}{n}< \varepsilon$ holds. Moreover, there exists an infinite sequence $F_m$ of $d_m$-regular $K_{s+2}$-saturated graphs on $n_m$ vertices such that $\frac{d_{m}}{n_{m}}=O_s(\frac{(\log\log n_{m})^2}{\log n_{m}})$ holds.
\end{theorem}

Let us turn now to $\regregsat(n,F)$ and $\overregsat(n,F)$. In both cases, the existence follows from the definition, thus we study their asymptotic behavior as $n$ gets large.

\begin{theorem}\label{regreg}~\\
\indent 
(i) For every $t\ge 1$ we have $$\regregsat(n,K_{t+2})=\Omega(n^{3/2}).$$

\smallskip 

(ii) For any $t\ge 1$, we have $$\liminf_{n \rightarrow \infty}\frac{\regregsat(n,K_{t+2})}{n^{3/2}}\le C_t$$ 

for some constant depending only on $t$.
\end{theorem}

\noindent 
Our main result on values of $\overregsat(n,F)$ is the following theorem.

\begin{theorem}\label{overreg}
For any $t\ge 1$, we have
$$\liminf_{n \rightarrow \infty} \frac{\overregsat(n,K_{t+2})}{n^{3/2}}=\sqrt{t}/2.$$
\end{theorem}

\bigskip 

\textbf{Notation.} For graphs $G$ and $H$, we introduce the blow-up of $G$ by $H$, denoted by $G[H]$, often called the lexicographic product of $G$ and $H$. To obtain $G[H]$, we replace every vertex of $G$ by a copy of $H$, and for every edge $uv$ of $G$, we add all the edges between the vertices of the corresponding copies of $H$.

\bigskip 

\textbf{Structure.} The rest of the paper is organized as follows. In the next section, we prove all the above results, while Section 3 contains further theorems on $\regsat(n,F)$, $\regregsat(n,F)$ and $\overregsat(n,F)$ for non-complete graphs $F$. Finally, Section 4 contains some concluding remarks.

\section{Complete graphs and general results}

We start this section by proving the general result of \begin{repproposition}{lower}~\\
\indent 
(i) Assume that every edge of $F$ is in a cycle of length at most $m+1$. Then $\regsat(n,F)= \Omega(n^{1+1/m})$.

(ii) If for any edge $e$ of $F$, the graph $F\setminus e$ has diameter at most $r$, then $\regsat(n,F)= \Omega(n^{1+1/r})$.

(iii) Given a connected graph $F$, there exists a constant $c=c_F$ such that ${\displaystyle\liminf_{n \rightarrow \infty}} \frac{\regsat(n,F)}{n}\le~c$ if and only if $F$ contains a cut edge.
\end{repproposition}

\begin{proof}
To prove (i) let us fix a vertex $v$ of a $d$-regular $F$-saturated graph $G$, it has $d$ neighbors and $n-d-1$ non-neighbors. As $G$ is $F$-saturated, and every edge of $F$ is in a cycle of length at most $m+1$, we have that for every non-neighbor $u$ there is a path of length at most $m$ from $v$ to $u$, but there are at most $d^m$ such paths, thus $n-d-1\le d^m$, $\regsat(n,F)\ge n^{1+1/m}$, and we are done. A very similar argument shows (ii).

If $F$ is 2-edge-connected, then (i) and (ii) both give superlinear lower bounds. Finally, if $F$ is connected but has a cut edge, then disjoint copies of cliques of size $|V(F)|-1$ show that for some $n$, we have a linear upper bound. 
\end{proof}

Now we can turn to existence results. As the proofs are not very hard, some details will be left to the reader.

\begin{reptheorem}{K3odd}
There exists an $n_0$ such that $\regsat(n,K_3)$ exists for every $n\ge n_0$. 

\end{reptheorem}

\begin{proof}[Proof of Theorem \ref{K3odd}]

First observe that if $n$ is even, then $K_{n/2,n/2}$ is a regular $K_3$-saturated graph. For odd $n$, we will define a graph $G_n$ on vertex set $V(G_n)=\{v_0,v_1,\dots,v_{n-1}\}$ and edge set $E(G_n)=\{(v_i,v_j):i-j\equiv a ~(\text{mod}\ n) ~\text{for some}\ a\in A_n\}$. The definition of $A_n$ depends on the modulo 10 residue class of $n$. In each case, it is left to the reader to see that $G_n$ is $K_3$-free and saturated if $n$ is large enough.

\medskip 

\textsc{Case I}. $n=10k+1=5(y+1)+1$.

Let $A_n=\{1,3,\dots,y\}\cup\{2y+2\}$.

\medskip

\textsc{Case II}. $n=10k+3=5(y+3)+3$.

Let $A_n=\{1,3,\dots,y\}\cup\{2y+2,2y+4, 2y+6,2y+8\}$.

\medskip

\textsc{Case III}. $n=10k+5=5(y+1)+5$.

Let $A_n=\{1,3,\dots,y\}\cup\{2y+2,2y+4\}$.

\medskip

\textsc{Case IV}. $n=10k+7=5(y+3)+7$.

Let $A_n=\{1,3,\dots,y\}\cup\{2y+2,2y+4, 2y+6,2y+8,2y+10\}$.

\medskip

\textsc{Case V}. $n=10k+9=5(y+1)+9$.

Let $A_n=\{1,3,\dots,y\}\cup\{2y+2,2y+4, 2y+6\}$.
\end{proof}

Before proving Theorem \ref{szubkvad}, we need a couple of lemmata. First we show how to build from regular $K_{s+2}$-saturated graphs a larger regular $K_{s+2}$-saturated graph.

\begin{construction}\label{amalgamation}
Let $H$ be a graph with vertex set $u_1,u_2,\dots,u_h$ and $G$ be a graph with vertex set $v_1,v_2,\dots, v_g$. Suppose $H$ contains $s-1$ pairwise edge-disjoint 2-factors. Furthermore, assume that these  2-factors can be oriented such that for any $c,d$ with $c+d=s+1$, the union of these oriented 2-factors does not contain a $K_{c,d}$ with all arcs oriented to the same part. Then for positive integers $s,t$ we define $H[s,t,G]$  as \textit{a} graph with vertex set 
$$V=\{u^j_i: 1\le j \le h, 1\le i\le t\}\cup \{v^b_a: 1\le b\le h,\, 1\le a\le g\}$$ 
and edge set $$E=\{u^j_iu^{j'}_{i'}: 1\le i,i'\le t, u_ju_{j'}\in E(H)\}\cup\{u^j_iv^j_a: 1\le j\le h, 1\le a\le g, 1\le i\le t\}\cup$$
$$\{v^b_av^{b'}_{a'}:1\le b<b'\le h,\, v_av_{a'}\in E(G)\}\cup \bigcup_{\ell=1}^{s-1}E_\ell,$$
where $E_\ell$ is defined as follows: let us orient all edges of the 2-factors the way described above, then $E_\ell$ contains all edges of the form $u^j_iv^{j'}_a$ with $1\le j,j'\le h$, $1\le i \le t$, $1\le a\le g$ and the edge $u_ju_{j'}$ is oriented towards $j'$ in the $\ell$th 2-factor.
\end{construction}

Note that $H[s,t,G]$ is not uniquely defined, as it also depends on the choice of the $s-1$ pairwise edge-disjoint 2-factors, and the choice of one of the orientations satisfying the desired oriented $K_{c,d}$-free property. However, the following lemma holds for any graph obtained this way.

\begin{lemma}\label{calc}~\\
\indent
(i) Assume that both $H$ and $G$ are $K_{s+2}$-saturated. Then for any $t \ge 1$, the graph $H[s,t,G]$ is $K_{s+2}$-saturated.

(ii) If $H$ is $d_H$-regular on $n_H$ vertices and $G$ is $d_G$-regular on $n_G$ vertices, then $H[s,t,G]$ is regular if and only if $st+(n_H-1)d_G=td_H+sn_G$. If so, then the regularity of $H[s,t,G]$ is $d=st+(n_H-1)d_G=td_H+sn_G$ and the number of its vertices is $n=n_H(t+n_G)$.
\end{lemma}

\begin{proof}
Both parts of the proof of (i) are by case analysis. Before, we need some definition. For fixed $1\le b\le h$, the independent set $\{v^b_a: 1\le a\le g\}$ is called the $G$-blob of $u_b$ . For fixed $1\le j\le h$, the independent set $\{u^j_i: 1\le i\le t\}$ is called the $H$-blob of $u_j$.

\smallskip 

To see that $H[s,t,G]$ is $K_{s+2}$-free:
\begin{itemize}
    \item 
    If $s+2$ vertices are either all in $G$-blobs or all in $H$-blobs, then they cannot form a $K_{s+2}$ as $G$ and $H$ are both $K_{s+2}$-free.
    \item
    If among $s+2$ vertices $x_1,x_2,\dots,x_{s+2}$ forming a $K_{s+2}$ in $H[s,t,G]$, there exist vertices from both $G$-blobs and $H$-blobs, then observe first that there can be at most one vertex $u$ of $H$ such that both an $H$-blob and a $G$-blob of $u$ are among the $x_i$s. Indeed, the 2-factors are pairwise edge-disjoint, therefore for any $j,j'$ the arc between $u_j,u_{j'}$ (if this edge exists at all in $H$) can be oriented in one way only, so either edges $u^{j}_av^{j'}_b$ or $u^j_bv^{j'}_a$ do not exist in $H[s,t,G]$.
    
    As both the $H$-blobs and the $G$-blobs of any vertex span an independent set in $H[s,t,G]$, therefore the $x_i$s belong to the blobs of at least $s+1$ distinct vertices of $H$. If the $x_i$s are to form a $K_{s+2}$ in $H[s,t,G]$, then the corresponding vertices of $H$ must form a $K_{c,d}$ with arcs oriented towards the part representing $G$-blobs. This contradicts the assumption on the union of the 2-factors, therefore $H[s,t,G]$ is indeed $K_{s+2}$-free.
\end{itemize}

To see that $H[s,t,G]$ is $K_{s+2}$-saturated, let $xy$ be an arbitrary non-edge of $H[s,t,G]$:
\begin{itemize}
    \item 
    If $x$ and $y$ belong to the same $H$-blob, then many copies of $K_{s+2}$ are created as $H$ is $K_{s+2}$-saturated and in any $K_{s+2}$-saturated graph, any vertex is contained in many copies of $K_{s+1}$. If $x$ and $y$ belong to the same $G$-blob, then a copy of $K_{s+2}$ is created as $G$ is $K_{s+2}$-saturated.
    \item
    If $x$ and $y$ belong to different $H$-blobs or to different $G$-blobs, then the existence of a $K_{s+2}$ in $H[s,t,G]\cup (xy)$ follows from the $K_{s+2}$-saturated property of $H$ and $G$, respectively.
    \item
    Finally, if $x$ is in an $H$-blob and $y$ is in a $G$-blob, then they are in blobs of different vertices $u_i,u_{i'}$ as the pairs from the $H$-blob and $G$-blob of the same vertex of $H$ already form edges in $H[s,t,G]$. So $x=u^{i}_b$ and $y=v^{i'}_a$ (with $i \neq i'$). Then as $G$ is $K_{s+2}$-saturated, $v_a$ is contained in a copy of $K_{s+1}$. Let $v_{b_1},v_{b_2},\dots,v_{b_s}$ be the other vertices of such a $K_{s+1}$. Furthermore let $u_{i_1},u_{i_2},\dots,u_{i_{s-1}}$ be the outneighbors of $u_i$ in the orientation of the $s-1$ many $2$-factors. Then $x$ and $y$ form a $K_{s+2}$ with $v^{i_1}_{b_1},v^{i_2}_{b_2}\dots,v^{i_{s-1}}_{b_{s-1}},v^{i}_{b_s}$.
    
\end{itemize}

The proof of (ii) is straightforward. The degree of a vertex in a $G$-blob is $st+(n_H-1)d_G$, while the degree of a degree in an $H$-blob is $td_H+sn_G$. In order to make $H[s,t,G]$ regular, these two quantities must be equal.
\end{proof}

\begin{lemma}
\label{multipartite}
For any positive integer $s$, there exists $q_0=q_0(s)$ such that for any $q\ge q_0$, the $(s+1)$-partite complete graph $K_{q,q,\dots,q}$ contains $s-1$ pairwise edge-disjoint 2-factors such that their union can be oriented the following way. For any positive integers $c,d$ with $c+d=s+1$, the union of these oriented 2-factors does not contain a $K_{c,d}$ with all arcs oriented towards the same part.
\end{lemma}

\begin{proof}
For any $j$ with $1 \le j \le s+1$ let $V^j=\{v^j_1,v^j_2,\dots,v^j_q\}$ be the vertex set of the $j$th partite set of $K_{q,q,\dots,q}$. Any integer $a\in \{0,1,\dots,q-1\}$ defines a natural oriented 2-factor on $K_{q,q,\dots,q}$ the following way. For $b=1,2,\dots,q$ the arcs $$(v^1_bv^2_{b+a}),(v^2_{b+a}v^3_{b+2a}),\dots,(v^{s}_{b+(s-1)a}v^{s+1}_{b+sa}),(v^{s+1}_{b+sa}v^1_b)$$ (where addition in the indices is modulo $q$) define oriented cycles, thus for each $b$ they are 2-factors with an orientation. Observe that if we take these 2-factors for each choice of $b$, they partition the vertex set. 

Let us consider the 2-factors corresponding to $a=0,1,\dots,s-2$. Suppose that $K$ is a copy of $K_{c,d}$ in the union of the 2-factors with all arcs oriented toward the same part. Then either vertices of one part are from $V^j$ and vertices from the other part are from $V^{j+1}$ for some $j=1,2,\dots,s+1$, where $(s+1)+1=1$.

By definition of the orientation of the 2-factors, the arcs in $K$ are oriented towards $V^{j+1}$. We distinguish two case.

\vskip 0.15truecm

\textsc{Case I.} $j\neq s+1$

\vskip 0.1truecm

Let $X=\{i:v^j_i \in V(K)\}$ and $Y=\{i:v^{j+1}_i\in V(K)\}$. As for any $x\in X$ the $K_{c,d}$ on $\{v^j_{i-x+1}:i\in X\}$ and $\{v^{j+1}_{i-x+1}:i\in Y\}$ have the orientation property, we may assume $1 \in X$. Because of the definition of the orientations, for any $x\in X,y\in Y$ we have $y-x\equiv l$ (mod $q$) for some $l\in \{0,1,\dots,s-2\}$. Then $Y\subseteq \{1,2,\dots,s-1\}$, and therefore $X\subset \{1,2,\dots, s-1\}\cup \{q-s+3,q-s+4,\dots,q\}$. Let $X^+=X\cap \{2,3,\dots,s-1\}$ and $X^-=X\cap \{q-s+3,q-s+4,\dots,q\}$, and let further $M$ be the maximal and $m$ be the minimal element of $Y$ (this time we consider the indices without modulus, so we have minimal and maximal elements). 

If $q\ge 4s$, then $\{y-1:y\in Y\}$, $\{m-x^+:x^+\in X^+\}$ and $\{M-x^-:x^-\in X^-\}$ are pairwise disjoint sets of representatives of residue classes mod $q$. Indeed, the elements of $\{m-x^+:x^+\in X^+\}$ are smaller than the elements of $\{y-1:y\in Y\}$, which are smaller than the elements of $\{M-x^-:x^-\in X^-\}$. Recall that the elements of these three sets
all belong to the residue classes of $\{0,1,\dots, s-2\}$. Therefore $|X|+|Y|=(|X^+|+|X^-|+|Y|)+1\le (s-1)+1=s$, and thus the sum of the part sizes of $K$ is at most $s$ as claimed.

\vskip 0.15truecm

\textsc{Case II} $j=s+1$

\vskip 0.1truecm

Let $X=\{i:v^1_i \in V(K)\}$ and $Y=\{i:v^{s+1}_i\in V(K)\}$. As for any $x\in X$ the $K_{p,q}$ on $\{v^1_{i-x+1}:i\in X\}$ and $\{v^{s+1}_{i-x+1}:i\in Y\}$ have the orientation property, we may assume $1 \in X$. Because of the definition of the orientations, we have $y-x\equiv sl$ (mod $q$) for some $l\in \{0,1,\dots,s-2\}$ for any $x\in X,y\in Y$. Then $Y\subseteq \{1,s+1,\dots,s(s-2)+1\}$, and therefore $X\subset \{1,s+1,\dots,s(s-2)+1\}\cup \{q-s(s-2)+1,q-s(s-3)+1,\dots,q-s+1\}$. Let $X^+=X\cap \{s+1,\dots,s(s-2)+1\}$ and $X^-=X\cap \{q-s(s-2)+1,q-s(s-3)+1,\dots,q-s+1\}$, and let further $M$ be the maximal and $m$ be the minimal element of $Y$. 

If $q\ge 4s^2$, then $\{y-1:y\in Y\}$, $\{m-x^+:x^+\in X^+\}$ and $\{M-x^-:x^-\in X^-\}$ are pairwise disjoint sets of representatives of residue classes mod $q$. Indeed, the elements of $\{m-x^+:x^+\in X^+\}$ are smaller than the elements of $\{y-1:y\in Y\}$, which are smaller than the elements of $\{M-x^-:x^-\in X^-\}$. Recall that the elements of these three sets
all belong to the residue classes of $\{0,s,\dots, s(s-2)\}$. Therefore $|X|+|Y|=(|X^+|+|X^-|+|Y|)+1\le (s-1)+1=s$, and thus the sum of the part sizes of $K$ is at most $s$ as claimed.
\end{proof}

\begin{lemma}\label{induction}
Assume $n=n_H(t+n_G)$ and $d=st+(n_H-1)d_G=td_H+sn_G$. If $\frac{d_G}{n_G}\ge \frac{s}{n_H-1}$, then $\frac{d}{n}\le \frac{n_H-1}{n_H}\cdot \frac{d_G}{n_G}$ holds.
\end{lemma}

\begin{proof}
We need to verify
\[
\frac{st+(n_H-1)d_G}{n_H(t+n_G)}\le \frac{n_H-1}{n_H}\cdot \frac{d_G}{n_G}.
\]
This is equivalent to 
\[
stn_Hn_G+n_H(n_H-1)d_Gn_G\le n_Ht(n_H-1)d_G+n_Hn_G(n_H-1)d_G.
\]
After cancelling terms on both sides and simplifying by $tn_H$, we obtain $sn_G\le (n_H-1)d_G$, which is equivalent to the condition of the lemma.
\end{proof}

\noindent 
Now we are ready to prove Theorem \ref{szubkvad}, which we restate here for convenience.

\begin{reptheorem}{szubkvad}
For any $\varepsilon > 0$ and integer $s \ge 1$ there exists a $d$-regular graph $F$ on $n$ vertices that is $K_{s+2}$-saturated and $\frac{d}{n}< \varepsilon$ holds. Moreover, there exists an infinite sequence $F_m$ of $d_{m}$-regular $K_{s+2}$-saturated graphs on $n_{m}$ vertices such that $\frac{d_{m}}{n_{m}}=O_s(\frac{(\log\log n_m)^2}{\log n_m})$ holds.
\end{reptheorem}

\begin{proof} Observe that the second part of the statement implies the first. We start with proving the first part to avoid unnecessary calculations and then we show how to modify the proof to obtain the second part of the statement.
Fix $\varepsilon>0$ and pick an arbitrary $d'$-regular $K_{s+2}$-saturated graph $G$ on $n'$ vertices. Pick $q$ large enough such that both Lemma \ref{multipartite} and $\frac{s}{2q-1}\le \varepsilon$ hold, and define $m$ to be the smallest integer such that $\frac{d'}{n'}\cdot (\frac{(s+1)q-1}{(s+1)q})^m\le \varepsilon$. Let $F_0=G[E_{(s(q-1))^m}]$, where $E_{(s(q-1))^m}$ is the empty graph on $(s(q-1))^m$ vertices, i.e., $F_0$ is the $(s(q-1))^m$ blow-up of $G$.

We define the following simple process. Assume a $d_i$-regular $K_{s+2}$-saturated graph $F_i$ on $n_i$ vertices is defined. If $\frac{d_i}{n_i}\le \varepsilon$, then $F_i=F$ is the desired graph. Otherwise using the notation of Construction \ref{amalgamation} we set $F_{i+1}=K_{q,q,\dots,q}[s,t,F_i]$ with an appropriately chosen $t$ to obtain a $d_{i+1}$-regular $K_{s+2}$-saturated graph on $n_{i+1}$ vertices, where $K_{q,q,\dots,q}$ has $s+1$ parts. According to Lemma \ref{induction}, if we can find a value $t$, then $\frac{d_{i+1}}{n_{i+1}}\le \frac{(s+1)q-1}{(s+1)q}\frac{d_i}{n_i}\le (\frac{(s+1)q-1}{(s+1)q})^{i+1}\frac{d_0}{n_0}=\frac{d'}{n'}\cdot (\frac{(s+1)q-1}{(s+1)q})^{i+1}$ holds. By definition of $m$, $F_m$ (or even some $F_j$ with $j<m$) will have $\frac{d_m}{n_m}\le \varepsilon$.

All we need to show is that an appropriate $t$ can be picked. Observe that in Construction~\ref{amalgamation} if $G$ and $H$ are fixed and regular, then to obtain $H[s,t,G]$ to be regular again, by Lemma~\ref{calc}, we need $st+(n_H-1)d_G=td_H+sn_G$. Equivalently, we need that $t=\frac{(n_H-1)d_G-sn_G}{d_H-s}$ is an integer. In our case $H$ is the complete $(s+1)$-partite graph $K_{q,q,\dots,q}$, thus $n_H=(s+1)q$ and $d_H=sq$. $G$ is $F_i$, thus $n_G=n_i$ and $d_G=d_i$.

We claim that our sequence $F_i$ of graphs satisfies that for any $i=0,1,\dots,m-1$, the values $d_i$ and $n_i$ are divisible by $(s(q-1))^{m-i}$. This is certainly true for $i=0$ as this is why we blew up $G$ by $(s(q-1))^m$ to obtain $F_0$. Then by induction on $i$, if $(s(q-1))^{m-i}$ divides $d_i,n_i$, then the corresponding $t$ value $\frac{((s+1)q-1)d_i-sn_i}{s(q-1)}$ is divisible by $(q-1)^{m-(i+1)}$. Therefore $d_{i+1}=st+((s+1)q-1)d_i$ and $n_{i+1}=(s+1)q(t+n_i)$ are both divisible by $(s(q-1))^{m-(i+1)}$. This concludes the proof of the first part of the theorem.

Finally, let $\varepsilon:=\frac{1}{q}$, then $m$ can be chosen as $(s+1)q\log q$ since $(\frac{(s+1)q-1}{(s+1)q})^{(s+1)q\log q}\le e^{-\log q}=\frac{1}{q}$. We need to calculate an upper bound on the number of vertices in the above construction. First observe that with $H=K_{q,q,\dots,q}$ we have $t_{i+1}=\frac{(n_H-1)d_i-sn_i}{d_H-s}\le \frac{(n_H-1)n_i-sn_i}{d_H-1}\le 2n_i$ and thus $n_{i+1}=n_H(t_{i+1}+n_i)\le 3n_Hn_i$. So for $F_m$ that contains $n_m$ vertices, we have $n_m \le (3\cdot (s+1)q)^mn_G\cdot (s(q-1))^m\le (3(s+1)q)^{2(s+1)q\log q}=:n$ vertices. As $\frac{(\log\log n)^2}{\log n}\ge \frac{(\log q-2\log\log q)^2}{2(s+1)q\log^2q}\ge \frac{1}{8sq}$ for large enough $q$ and all $s\geq 1$, we have that $\varepsilon=O(\frac{(\log\log n)^2}{\log n})$, thus the result follows. 
\end{proof}

Now we turn our attention to the proof of

\begin{reptheorem}{regreg}~\\
\indent 
(i) $\regregsat(n,K_{t+2})=\Omega(n^{3/2})$ for $t\ge 1$.

(ii) For any $t\ge 1$, we have ${\displaystyle\liminf_{n \rightarrow \infty}} \ \frac{\regregsat(n,K_{t+2})}{n^{3/2}}\le C_t$ for a constant depending only on~$t$.
\end{reptheorem}

\begin{proof}
To see the lower bound of (i), observe that if $n$ is even and $G$ is a $K_{t+2}$-free $d$-regular graph with $1+d^2<n/2$, then for every vertex $v$, there are more than $n/2$ vertices at distance more than 2. Let us define an auxiliary graph $G'$ the following way. Let $V(G')=V(G)$ and $uv$ is an edge in $G'$ if and only if $d_G(u,v)\ge 3$. By Dirac's theorem, $G'$ contains a Hamiltonian cycle. We claim that one can add every other edge of this Hamiltonian cycle (thus a perfect matching) to $G$ to obtain a $(d+1)$-regular $K_{t+2}$-free graph $G^*$. Indeed, as the added edge set is a matching, a triangle can only contain one of its edges. By definition, if $uv$ is such an edge, then $d_G(u,v)\ge 3$. Therefore, $uv$ cannot be contained in a triangle in $G^*$, which implies it cannot be contained in a $K_{t+2}$ in $G^*$. 

If $n$ is odd, we again use the graph $G'$. This time we assume that $1+d^2<cn$ for $c=\frac{1}{3}$. This implies that $G'$ has minimum degree at least $\frac{2}{3}n$. By a theorem of Koml{\'o}s, S{\'a}rk{\"o}zy and Szemer{\'e}di \cite{kssz} (weaker constants were proved earlier in \cite{fh,fk,fgjs}), $G'$ contains the square of a Hamiltonian cycle $C$ if $n$ is large enough. We claim that the graph $G^*$ that we obtain by adding the Hamiltonian cycle $C$ to $G$ does not contain any new triangle. As we added a cycle of length larger than 3, we could not add all three edges of a new triangle. As we added only edges of $G'$, it cannot happen that a triangle with exactly one new edge is created. Finally, adjacent new edges belong to $C$, and for such pairs of edges, the third edge making the triangle complete belongs to $G'$ as well, so this triangle does not exist in $G^*$. 

In proving the upper bound of (ii), our strategy will be to define a regular $K_{t+2}$-free graph $G$ that contains a vertex $v$ such that adding any non-edge of $G$ incident to $v$ would create a $K_{t+2}$. Clearly, such a graph is  $K_{t+2}$-free and adding edges to all vertices would create a copy of $K_{t+2}$.

For fixed $t$ and arbitrary even $d$, we define a $dt$-regular graph $G$ as follows: the induced subgraph of $G$ on the neighborhood $N(v)$ of the special vertex $v$ is the union of $d$ cliques of size $t$. Let these cliques be $A_1,A_2,\dots,A_d$. Let $B_1,B_2,\dots,B_d$ be pairwise disjoint sets of size $dt-t$. Let us join every vertex $u_i\in A_i$ to all vertices in $B_i$. In this way, $v$ and all vertices in $\cup_{i=1}^dA_i$ have degree $dt$ in $G$. Finally, we add on $\cup_{=1}^d B_i$ an arbitrary $(dt-t)$-regular bipartite graph such that the parts are $\cup_{i=1}^{d/2}B_i$ and $\cup_{i=d/2+1}^{d}B_i$. 

By definition, $G$ is $dt$-regular. It does not contain any clique of size $t+2$ as if the clique contains $v$, then $N(v)$ contains cliques only of size at most $t$, while any pair of vertices $b_1,b_2\in \cup_{i=1}^d B_i$ joined by an edge belong to different $B_i$'s therefore they do no share common neighbors. 

As claimed before, any non-edge of $G$ incident to $v$ creates a copy of $K_{t+2}$. Indeed, such an edge has an endpoint $b\in B_i$ for some $i=1,2,\dots,d$ and then $v,b$ and the vertices of $A_i$ form a clique of size $t+2$. The number of vertices in $G$ is $1+(dt)^2$, while the regularity of $G$ is $dt$.

\end{proof}

We finish this section with the proof of our result on the oversaturation number of complete graphs.

\begin{reptheorem}{overreg}
For any $t\ge 1$, we have
${\displaystyle\liminf_{n \rightarrow \infty}} \frac{\overregsat(n,K_{t+2})}{n^{3/2}}=\sqrt{t}/2$.
\end{reptheorem}

\begin{proof}
The lower bound follows from the fact that in a $d$-regular $K_{t+2}$-oversaturated graph $G$ there must be at least $t$ paths of 2 edges from any vertex to any of its non-neighbors, thus $d(d-1)\ge t(n-d-1)$ should hold.

For the upper bound, let us consider first the case $t=1$.  We obtain the following construction with a little alteration of the well-known polarity graph: let $\F$ be the field of size $2^p$. Let us define the equivalence relation $(a,b,c)\sim (x,y,z)$ over all triples of $\F$ by two triples being in relation if there exists a non-zero $u\in \F$ with $au=x,bu=y,cu=z$. Let $G'$ be the graph having the set of not all-zero equivalence classes as vertex set and $(a,b,c)$ adjacent to $(x,y,z)$ if $ax+by+cz=0$. The number of vertices of $G'$ is $2^{2p}+2^p+1$. As the system of linear equations $ax+by+cz=0=a'x+b'y+c'z$ is uniquely solvable for all non-all-zero $(a,b,c)$ and $a',b',c'$, the diameter of $G'$ is 2. All degrees are either $2^p$ or $2^p+1$ depending on whether for the triple $(a,b,c)$ we have $a^2+b^2+c^2=0$ or not. Moreover, it is known 
that, writing $e$ for the multiplicative unit of $\F$, the set of triples having degree $2^p$ is exactly $S=\{(e,x,e-x):x\in \F\}\cup \{(0,e,e)\}$. Indeed, low degree vertices are those satisfying $x^2+y^2+z^2=0$. As the characteristic of the field is 2, we have $(x+y+z)^2=x^2+y^2+z^2=0$ if and only if $x+y=z$. Therefore there are $(2^p)^2$ solutions one of which is $(0,0,0)$ and the others are the $2^p+1$ equivalence classes of $S$. Moreover, $S$ is the set of neighbors of the triple $(e,e,e)$. Let us define $G$ to be the graph obtained from $G'$ by adding a new vertex $v$ as a twin of $(e,e,e)$, i.e. $(e,e,e)$ and $v$ are not adjacent, and $v$ and $(x,y,z)$ are joined by an edge in $G$ if $(e,e,e)$ and $(x,y,z)$ are joined by an edge in $G'$. By the above, $G$ is $(2^p+1)$-regular on $2^{2p}+2^p+2$ vertices. Clearly, $d_G(v,(e,e,e))=2$ and for $(a,b,c)\neq (e,e,e)$ we have $d_G(v,(a,b,c))=d_{G'}((e,e,e),(a,b,c))$, thus $G$ has also diameter 2 and therefore $K_3$-oversaturated.

For general $t$, let us consider the $t$-blow-up $G_t$ of $G$; i.e., $G_t=G[K_t]$. Observe that $G_t$ is $K_{t+2}$-oversaturated as if $x,y$ form a non-edge in $G_t$, then the vertices $u,v$ of $G$ corresponding to $x,y$ have a common neighbor $w$. Then $x,y$ and the $t$ vertices in $G_t$ corresponding to $w$ form a $K_{t+2}$. Clearly, we have $|V(G_t)|=t|V(G)|$ and $d_{G_t}=td_G+t-1$.
\end{proof}

\section{Other graphs}

\begin{theorem}\label{rizsfelfujt}
Let $F$ be an edge-transitive graph and let $G$ be  regular $F$-saturated. Then for any $t\ge 2$, the graph $G[K_t]$ is regular $F'_t$-oversaturated, where $F'_t$ is obtained from $F$ by removing an edge to get $F'$, then adding back an edge to $F'[K_t]$ to obtain $F'_t$. Moreover, if $F=K_s$ for some $s$, then $G[K_t]$ is $F'_t$-saturated. In particular, for any $t\ge 2$ and $F=K_s$, we have ${\displaystyle\liminf_{n \rightarrow \infty}} \frac{\regsat(n,F'_t)}{n^2}=0$.
\end{theorem}

\begin{proof}
Let $F$ be edge-transitive, $G$ be $F$-saturated. Clearly, if $G$ is regular, then so is $G[K_t]$. Let $xy$ be a non-edge of $G[K_t]$. Then $uv$ is a non-edge of $G$, where $u$ and $v$ are the vertices corresponding to $x$ and $y$ in $G$. As $G$ is $F$-saturated there is a copy of $F$ in $G \cup xy$ on vertex set $S \ni x,y$. Then, as $F$ is edge-transitive, $G[K_t]$ spans a graph on the blow-up of $S$ that contains $F'_t$.

Suppose now $F=K_s$ for some $s\ge 3$. We need to show that $G[K_t]$ is $(K_s)'_t$-free. Suppose to the contrary that $G[K_t]$ contains a copy of $(K_s)'_t$. Observe that $(K_s)'_t$ contains two cliques of size $(s-1)t$ intersecting in $(s-2)t$ vertices. We claim that these cliques in $G[K_t]$ must be unions of $s-1$ blow-ups each. Indeed, $(s-1)t$ vertices must meet at least $s-1$ blow-ups and as meeting $s$ blowups would yield that $G$ contains a $K_s$ contrary to our assumption, we obtain that the cliques are indeed unions of $s-1$ blow-ups. The extra edge of $(K_s)'_t$ present in $G[K_t]$ would show that a corresponding edge is present in $G$, again yielding a $K_s$ in $G$. This contradiction proves that $G[K_t]$ is indeed $(K_s)'_t$-free.
\end{proof}

Let $F_6$ denote the graph on vertices $a,b,c,d,e,f$ with edges $ab,ac,bc,ad, bd, be, ce, af,cf$. It is sometimes referred to as the 3-sun graph.

\begin{theorem}\label{3sun}
If $G$ is a regular, $K_3$-saturated graph, then $G[K_2]$ is regular $F$-saturated for any $F_6\subseteq F \subseteq (K_3)'_2$. In particular, ${\displaystyle\liminf_{n \rightarrow \infty}} \frac{\regsat(n,F)}{n^2}=0$ for all such graphs.
\end{theorem}

\begin{proof}
Observe that the triangles $abd,bce,acf$ in $F_6$ are pairwise edge-disjoint, pairwise vertex-intersecting with no vertex contained in all three of them. As $G$ is $K_3$-free, all triangles in $G[K_2]$ belong to an original edge. If three triangles of $G[K_2]$ belong to the same original edge, then they are contained in a $K_4$, so clearly cannot have the above property. If two of the triangles belong to disjoint original edges, then the triangles are disjoint. As $G$ is triangle-free, either two of the triangles must come from the same original edge and the third from an adjacent edge or all three triangles must come from three original edges sharing the same vertex of $G$. In both cases, one cannot have that the three triangles pairwise intersect, but they do not share a common vertex. Therefore $G[K_2]$ is indeed $F_6$-free.

On the other hand, as $G$ is $K_3$-saturated, adding any edge to $G[K_2]$ would create a copy of $(K_3)'_2$ by Theorem \ref{rizsfelfujt}.
\end{proof}

\section{Concluding remarks}

In this paper we studied regular saturation, but we are left with more questions than answers. The main question left wide open is the following. For which graphs do we have that $\regsat(n,F)$ exists for every $n$ large enough? We could show the existence for the triangle and some other graphs, but we do not know the answer even for the most natural graph classes such as paths, cycles, cliques. To motivate some further research let us put here some observations on the existence of regular $K_4$-saturated graphs for certain infinite sequences of $n$.

\begin{proposition}\label{k4}
There exist regular $K_4$-saturated graphs of order $n$ for the following residue classes:
\begin{itemize}
    \item[$(i)$] $n\equiv 6$~{\rm (mod 8)},
    \item[$(ii)$]
     $n\equiv 0$~{\rm (mod 8)},
    \item[$(iii)$] $n\equiv 0$~{\rm (mod 17)}.
\end{itemize}
\end{proposition}

\emph{Sketch of proof.}~\\
\indent 
$\bullet$ To prove (i) for $n=8k+6$ let us define a graph $G_n$ on the vertex set $V(G_n)=\{v_0,v_1,\dots,v_{n-1}\}$ and edge set $E(G_n)=\{(v_i,v_j):i-j\equiv a ~\text{for some}\ a\in \{1, 2, 5, 6, ..., 4k+ 1, 4k+ 2 \}\}$. One can easily see that this construction works.

\medskip 

$\bullet$ To prove $(ii)$ and $(iii)$ we use the following claim the verification of which is left to the reader.

\begin{claim}\label{join}
If $G$ is $d_G$-regular $K_s$-saturated and $H$ is $d_H$-regular $K_t$-saturated, then the join $G+H$ is $K_{s+t-1}$-saturated and if $d_H+v_G=d_G+v_h=:d$, then $G+H$ is $d$-regular.
\end{claim}

The proof of (ii) follows from Claim \ref{join} applied to $G=C_5[E_k]$ and $H=E_{3k}$ with $s=3,t=2$.

The proof of (iii) follows from Claim \ref{join} applied to $G=P[E_k]$ and $H=E_{7k}$ with $s=3,t=2$, where $P$ is the Petersen-graph.
\qed

\bigskip

Let us consider what happens if we ask weaker questions than whether there exists a regular $F$-saturated graph for every large enough $n$:

\medskip 

For which graphs does $\regsat(n,F)$ exist for infinitely many values of $n$? For trees and cliques this existence is obvious, while it is not hard to see that for a matching $M_k$ of $k$ edges, $\regsat(n,F)$ does not exists if $n>4k-4$. Indeed, in an $M_k$-free graph $G$, $2k-2$ vertices cover all the edges, thus the sum of their degrees is at least $|E(G)|$, which is at least the sum of the degrees of the other vertices.

\smallskip 

For which graphs does $\regsat(n,F)$ exist for some $n$? 
Obviously, if $n<|V(F)|$, then $K_n$ is regular and $F$-saturated, thus the meaningful question assumes $n\ge |V(F)|$. Even that does not hold for every graph, although the only counterexample we are aware of is the matching $M_2$.

\bigskip 

Another variant is if we lessen our requirements on regularity:

\smallskip 

Let us call a graph \textit{almost regular} if every degree is $d$ or $d+1$ for some $d$. Then it is immediate that there exists an almost regular $F$-saturated $n$-vertex graph for every $n$ if $F$ is a tree or clique. However, the matching $M_k$ is again an example where $n$ needs to be small for the existence of $G$.
Let us call a graph $(d_1,d_2)$-biregular if every degree is $d_1$ or $d_2$. Maybe this is the weakening of the conditions that is enough to ensure the existence for every $n$. This is the case at least for $M_k$, as there is an $M_k$-saturated graph for $n\ge 2k$ with $k-1$ vertices of degree $n-1$ and $n-k+1$ vertices of degree $k-1$.

\section*{Acknowledgement}

Research was supported by the National Research, Development and Innovation Office - NKFIH under the grants FK 132060, KH130371, KKP-133819 and SNN 129364. Research of  Vizer was supported by the J\'anos Bolyai Research Fellowship. and by the New National Excellence Program under the grant number \'UNKP-20-5-BME-45.
Patk\'os acknowledges the financial support from the Ministry of Educational and Science of the Russian Federation in the framework of MegaGrant no. 075-15-2019-1926.

\end{document}